\documentclass[twoside,12pt]{article}
\date{\today}
\usepackage{amsmath,amsthm,amsfonts,latexsym,amssymb,fancybox,comment,graphicx,subfig,array,subeqnarray,indentfirst}
\usepackage{hyperref,mathrsfs}
\usepackage[mathscr]{euscript}

\date{\it Dedicated to Prof. Jim Douglas, Jr. on the
  occasion of his eighty fifth birthday.}

\title{A class of nonparametric DSSY nonconforming quadrilateral elements
}\author{Youngmok Jeon\thanks{Department of Mathematics,
         Ajou University,
         Suwon 443--749,
         Korea;
E-mail: yjeon@ajou.ac.kr
} ~~~Hyun Nam\thanks{Department of Mathematics,
         Seoul National University,
         Seoul 151--747,
         Korea;
E-mail: lamyun96@snu.ac.kr
}~~~ Dongwoo Sheen\thanks{Department of Mathematics, and
Interdisciplinary Program in Computational Science \& Technology,
         Seoul National University,
         Seoul 151--747,
         Korea
E-mail: sheen@snu.ac.kr
} ~~~Kwangshin Shim\thanks{Department of Mathematics,
         Seoul National University,
         Seoul 151--747,
         Korea
E-mail: sim4322@snu.ac.kr;\newline \newline
This paper will appear in  {\bf ESAIM--Math. Model. Numer. Anal.}
}}

\numberwithin{equation}{section}

\newtheorem{theorem}{Theorem}[section]

\newtheorem{remark}[theorem]{Remark}

\newcommand{\rmkref}[1]{Remark~\ref{#1}}

\newcommand{\figref}[1]{Figure~\ref{#1}}

\newcommand{\tabref}[1]{Table~\ref{#1}}
\newcommand{\tabrefs}[2]{Tables~\ref{#1} and ~\ref{#2}}
\newcommand{\tabrefss}[2]{Tables~\ref{#1} -- ~\ref{#2}}

\newcommand{\eq}[1]{\begin{eqnarray}\label{#1}}
\newcommand{\qe}{\end{eqnarray}}

\newcommand{\be}{\begin{eqnarray}}
\newcommand{\ee}{\end{eqnarray}}
\newcommand{\bal}{\begin{aligned}}
\newcommand{\eal}{\end{aligned}}
\newcommand{\bes}{\begin{eqnarray*}}
\newcommand{\ees}{\end{eqnarray*}}
\newcommand{\bs}{\begin{subeqnarray}}
\newcommand{\es}{\end{subeqnarray}}
\newcommand{\bss}{\begin{subeqnarray*}}
\newcommand{\ess}{\end{subeqnarray*}}

\newcounter{saveeqn}

\renewcommand{\hat}{\widehat}
\renewcommand{\tilde}{\widetilde}

\def\Span{\operatorname{Span}}

\def\hK{\widehat K}

\def\til{\widetilde}

\def\ts{\widetilde s}

\def\diam{\operatorname{diam}}

\def\sq{\mathscr{Q}}

\def\O{\Omega}
\def\cA{\mathcal{A}}
\def\cS{\mathcal{S}}
\def\cE{\mathcal{E}}
\def\cF{\mathcal{F}}

\def\Tau{\mathcal{T}}
\def\cF{\mathcal{F}}
\def\NC{\mathcal{NC}}
\def\DSSY{\mathcal{NC}^{DSSY}}

\def\and{\quad\text{and}\quad}

\def\<{\left\langle}
\def\>{\right\rangle}

\def\bc{{\boldsymbol \xi}_0}

\def\bv{\mathbf v}

\def\bx{\mathbf x}
\def\bb{\mathbf b}
\def\bd{\mathbf d}

\def\bx{\mathbf{x}}
\def\hx{\hat{x}_1}
\def\hy{\hat{x}_2}
\def\hbm{\hat{\mathbf{m}}}

\def\hbx{\hat{\mathbf{x}}}
\def\hbu{\hat{\mathbf{u}}}
\def\hbv{\hat{\mathbf{v}}}
\def\tx{\tilde{x}_1}
\def\ty{\tilde{x}_2}
\def\tbg{\tilde{\mathbf{g}}}
\def\bc{\mathbf{c}}

\def\tbeta{\tilde{\boldsymbol{\eta}}}
\def\tbe{\tilde{\mathbf{e}}}
\def\tbm{\tilde{\mathbf{m}}}
\def\tbv{\tilde{\mathbf{v}}}
\def\tbx{\tilde{\mathbf{x}}}
\def\tbs{\tilde{\mathbf{s}}}

\def\bv{\mathbf{ v}}

\usepackage{colortbl}

\def\cor#1{#1}
\def\addcor#1{#1}
\newcolumntype{x}[1]{>{\centering\hspace{0pt}}p{#1}} 

\begin{document}
\maketitle
\begin{abstract}
A new class of nonparametric nonconforming quadrilateral finite elements is introduced
which has the midpoint continuity and the mean value continuity at the
interfaces of elements simultaneously as the rectangular DSSY element \cite{dssy-nc-ell}.
The parametric DSSY element for general quadrilaterals requires five degrees
of freedom to have an optimal order of convergence \cite{cdssy},
while the new nonparametric DSSY elements require only four degrees of freedom.
The design of new elements is based on
the decomposition of a bilinear transform into a simple bilinear map followed
by a suitable affine map.
Numerical results are presented to compare the new elements with the parametric DSSY element.
\end{abstract}


\section{Introduction}
There have been many progresses for nonconforming finite
element methods for many mechanical problems for last decades.
Nonconforming elements have been a favorite choice in solving the Stokes and Navier-Stokes
equations \cite{cr73, cdy99, cheapest-nc-rect, rann, turek} in a stable manner. Also,
the nonconforming nature
facilitates resolving numerical locking \cite{brenner-sung,
lls-nc-elast, zhiminzhang} in elasticity problems \cor{with the clamped
  boundary condition.
For pure traction boundary value problems in elasticity, there have been
a couple of
approaches to avoid numerical locking by employing conforming and
nonconforming elements componentwise \cite{kouhia-stenberg, ming-shi-reissner,
ming-shi-two}.
} \cor{Although there are several higher-order nonconforming elements,}
the lowest order nonconforming elements have been especially popular numerical methods
because of its simplicity and stability property
\cite{cr73, rann, turek}.
\cor{In particular, the linear simplicial nonconforming elements introduced by Crouzeix
  and Raviart \cite{cr73} have been most widely used.
Since the degrees of freedom for quadrilateral or rectangular elements are usually smaller
than those for triangular elements, it is desirable to use
quadrilateral or rectangular elements wherever they can be applied.}

We briefly review some progresses for nonconforming rectangular or quadrilateral
elements.
Han introduced firstly a rectangular element which assumes five local degrees
of freedom (DOFs) \cite{han84} in 1984.
Then in 1992 Rannacher and Turek introduced the rotated $Q_1$ nonconforming elements with
two types of degrees of freedom \cite{rann}: the four edge-midpoint
value DOFs and the four edge integral DOFs.
Chen \cite{chen-projection-93} also used the first type of DOFs for the same
rotated $Q_1$ element.
Douglas, Santos, Sheen and Ye introduced a new nonconforming finite element,
which we call the {\it DSSY element} in this paper,
for which the two types of degrees of freedom are coincident on rectangular
(or parallelogram) meshes \cite{dssy-nc-ell}. \cor{One of the key features of
this DSSY element is that it fulfills the {\it mean value property}
on each edge. For a convergence analysis, the average continuity
property over each edge implies the pass of ``patch test'', which is a
sufficient condition for optimal convergence of nonconforming finite element
methods \cite{shi-stummel-patch, shi-fem-patch, wang-patch}.
Notice that using the edge-midpoint values is not only cheaper but also simpler
than using edge-integral values in constructing the local and global basis
functions.
For instance, in gluing two neighboring elements across an edge, only one
evaluation
at the edge midpoint is necessary for the DSSY-type element while at least two
Gauss-point evaluations are necessary for the elements using integral type DOFs.
Therefore, nonconforming elements fulfilling the mean value property have advantages
in implementation. The Crouzeix-Raviart $P_1$-nonconforming elements
\cite{cr73} enjoy the mean value property.}

Arnold, Boffi, and Falk provided a theory of convergence order in
 quadrilateral meshes \cite{ABF}. A  modified DSSY element
 was introduced in \cite{cdssy}, which requires an additional
 DOF in order to retain an optimal convergence order for genuinely
 quadrilateral meshes. It seems impossible to reduce the number of DOFs from
five to four
as long as one considers a parametric DSSY-type element on
quadrilateral meshes and still wants to preserve optimal convergence.

The aim of this paper is to attempt to extend the spirit of rectangular DSSY
element to genuinely quadrilateral meshes keeping the mean value property with
four DOFs, shifting from the parametric realm to the nonparametric one.
Our starting point is based on a clever decomposition of a bilinear map
into a simple bilinear map followed by an affine map \cite{koster2012new,parksheen-p1quad, park-sheen-morley}.
This approach induces an
intermediate reference quadrilateral, where a four DOF DSSY-type element can
be defined.
Then the affine map will preserve $P_1$ and the mean value property on each edge.
We remark that the quadrilateral element introduced in \cite{parksheen-p1quad}
is of only three DOFs, and a similar element was introduced by Hu and Shi
\cite{hu-shi}, but without any modification they cannot be used to solve fluid
and solid mechanics in a stable manner.
%

The paper is organized as follows.
In section 2 we review some specific properties of the DSSY element.
Then using the decomposition of a bilinear map into a simple
bilinear map followed by an affine map, we introduce a family of
quadrilateral elements on an intermediate reference quadrilateral, which
is of four DOFs. Based on this, we define a family of nonparametric
quadrilateral elements.
Section 3 is devoted to numerical experiments. The performance of
the new nonparametric DSSY elements and the parametric DSSY element is
compared in terms of computation time where the nonparametric DSSY elements
show a clear advantage over the parametric one.

\section{Quadrilateral nonconforming elements}
In this section we will introduce a nonparametric DSSY element of four local
degrees of freedom. First of all let us review the (parametric) DSSY element
in brief.
\subsection{The DSSY element}
Let $\Omega$ be a simply connected polygonal domain in $\mathbb{R}^2$ and
$(\Tau_h )_{h>0} $ be a family of shape regular quadrilateral
triangulations of $\Omega$ with $\max_{K \in \Tau_h}
\diam(K)=h$. Let us denote by $\cE_h$  the set of all edges in
$\Tau_h$. For an element $K \in \Tau_h $ we denote four vertices
of $K$ by $\mathbf{v}_j$ for $j=1,2,3,4.$ Also denote the edge passing
through $\mathbf{v}_{j-1}$ and  $\mathbf{v}_j$ by $e_j$ and the midpoint of
$e_j$ by $\mathbf{m}_j$ for  $j=1,2,3,4,$ (assuming $\mathbf{v}_{0} := \mathbf{v}_{4}$,)
as in \figref{fig:Bil}.
The linear polynomials $l_{13}$ and $l_{24}$ are defined in a way that two line equations $l_{13}=0$, $l_{24}=0$ pass through $\mathbf{m}_1$, $\mathbf{m}_3$, and $\mathbf{m}_2$, $\mathbf{m}_4$, respectively.
 Consider a reference square $\hat{K} = [-1,1]^2 $. We use the similar notations for vertices, edges, midpoints of $\hat{K}$ as those of $K$ such as $\hat{\mathbf{v}}_{j}$, $\hat{e}_j$, and $\hat{\mathbf{m}}_j$ for $j=1,2,3,4$.

Let $K\in \Tau_h $ be any quadrilateral. Then
there exists a bilinear map $\cF_K: \hK \rightarrow K$ such
that $\cF_K(\hK) = K.$
Notice that $\cF_K$ can be written as follows:
\begin{equation}\label{eq:bilinear}
\cF_K(\hbx)=\mathbf{v}_1 +
\frac{1-\hx}{2}(\mathbf{v}_2-\mathbf{v}_1) +\frac{1-\hy}{2}
(\mathbf{v}_4-\mathbf{v}_1) + \frac{(1-\hx)(1-\hy)}{4} (\mathbf{v}_1 -\mathbf{v}_2 + \mathbf{v}_3 -\mathbf{v}_4 ).
\end{equation}

Set
\begin{equation*}
\DSSY_{\hat{K},l} =\{ 1, \hx, \hy, \hat{\varphi}_l(\hx) -\hat{\varphi}_l(\hy) \}, \quad l=1,2,
\end{equation*}
where
\begin{equation} \label{DSSY-psi}
\hat{\varphi}_l(t)= \left\{
\begin{array}{ll}
t^2 -\frac{5}{3}t^4, & l=1, \\
\\
t^2 -\frac{25}{6}t^4 + \frac{7}{2}t^6,& l=2.
\end{array} \right.
\end{equation}

Then the degrees of freedom for the DSSY element can be chosen as either
four mean values over edges or four edge-midpoint values, which turn out to be
identical. In other words, the DSSY elements fulfill the {\it mean value property}:
\begin{equation}\label{eq:mvp}
\frac{1}{|\hat{e}_j|} \int_{\hat{e}_j} \hat{v}~ d\hat{\sigma}
=\hat{v}(\hat{\mathbf{m}}_j), ~~~~~   j=1,2,3,4,\quad\forall \hat{v}\in \DSSY_{\hat{K},l}.
\end{equation}

In order to retain an optimal convergence order for any quadrilateral mesh,
the parametric DSSY element needs an additional element $\hx\hy$, and
therefore the modified reference element reads
\begin{equation*}
{\DSSY_{\hat{K},l}}^* =\{ 1, \hx, \hy, \hx\hy, \hat{\varphi}_l(\hx) -\hat{\varphi}_l(\hy) \}, \quad l=1,2,
\end{equation*}
with an additional degree of freedom
\begin{equation*}
\int_{\hat{K}} \hat{v}(\hbx)\hx\hy~\operatorname{d\hx d\hy}.
\end{equation*}

The DSSY element on $K$ is then defined by
\begin{equation*}
\DSSY_{K}=\begin{cases}
\{ v ~|~v=\hat{v} \circ \cF_K^{-1}, \hat{v} \in {\DSSY_{\hat{K},l}}^*  \}
&\quad\text{if } K \text{ is a true quadrilateral},\\
\{ v ~|~v=\hat{v} \circ \cF_K^{-1}, \hat{v} \in \DSSY_{\hat{K},l}  \}
&\quad\text{if }K \text{ is a rectangle},
\end{cases}
\end{equation*}
where  $\cF_K$ is defined by \eqref{eq:bilinear}.
The global parametric DSSY element is defined by
\begin{equation*}
\begin{split}
&\NC^{p}_{h}=\{ v_h \in L^2(\O)  ~|~ v_h|_K \in \DSSY_K ~\mathrm{for}~ K \in \Tau_h, v_h ~\mathrm{is}~ \mathrm{continuous}~ \mathrm{at}~ \mathrm{the}~ \mathrm{midpoint}~ \mathrm{of}~ \mathrm{each}~e \in \cE_h    \},   \\
&\NC^{p}_{h,0}=\{v_h \in \NC^{p}_{h} ~|~v_h~\mathrm{is}~ \mathrm{zero}~ \mathrm{at}~ \mathrm{the}~ \mathrm{midpoint}~ \mathrm{of}~ e \in \cE_h \cap \partial \O \}.
\end{split}
\end{equation*}

\subsection{A Class of Nonparametric DSSY Elements}
We are interested in reducing the five degrees of freedom DSSY element
to four, but still retaining the {\it mean value property} \eqref{eq:mvp}.
It seems that there does not exist a four-DOF parametric quadrilateral
element which has an optimal order convergence rate and the {\it mean value
  property} simultaneously.
Here, we seek a candidate among nonparametric elements.

\subsubsection{A closer look at the DSSY element}
\addcor{For the sake of simplicity of our argument regrading  the geometrical property of a basis function,
we shall focus on}, $\hat{\varphi}_1(\hx) -\hat{\varphi}_1(\hy)$.

Let us denote $\hat{\varphi}_1(\hx) -\hat{\varphi}_1(\hy)$ by $\hat{\psi}(\hbx)$  for convenience.
In the reference domain $\hat{K}$, the function $\hat{\psi}(\hbx)$
can be factorized as
\begin{equation}\label{eq:psi-parametric}
\hat{\psi}(\hbx)=-\frac{5}{3}(\hx-\hy)(\hx+\hy)(\hx^2+\hy^2-\frac{3}{5} ),
\end{equation}
from which one can realize that $\hat{\psi}(\hbx)$ is the
product of three polynomials whose zero-level sets consist of
the two diagonals of $\hat{K}$ and one circle
$\hx^2+\hy^2-\frac{3}{5}=0$ in $\hat{K}.$
At this point, a natural question is
whether for any quadrilateral $K$ we may find a function satisfying the {\it
  mean value properties} by using the similar geometrical idea as
$\hat{\psi}(\hbx)$.

Among the parametric nonconforming elements in \cite{dssy-nc-ell}, $\psi(\bx)=
\hat{\psi} \circ \cF_K^{-1}$ is not a quartic polynomial in general
if $K$ is a genuine quadrilateral, that is, if $\cF_K$ is not an affine map.
In most cases it is a non-polynomial function. Thus $\psi(\bx)$ would not be
similarly regarded as the product of zero level set functions of
three geometrical objects, such as two lines and a circle. This seems to be
one of the limits of using parametric elements. We will thus divert \addcor{our
attention} from using the parametric elements and investigate \addcor{a possible way of finding} a suitable four
degrees of freedom element.

\subsubsection{Intermediate Spaces}
To design such a suitable element, we first decompose the bilinear map
$\cF_K$ given by \eqref{eq:bilinear} into a composition of a {\it simple bilinear map} followed by
an affine map \cite{koster2012new,parksheen-p1quad, park-sheen-morley}. A bilinear map $S: \mathbb R^2\to \mathbb
R^2$ is said to be a {\it simple bilinear map} if there exists a vector
$\tilde{\mathbf s}$
such that
$S {x_1\choose x_2} = {x_1\choose x_2} +
x_1x_2 \tilde{\mathbf s}$ for all $ {x_1\choose x_2} \in \mathbb R^2.$

Observe that $\cF_K$ can be written as follows:
\begin{eqnarray}\label{eq:cFK}
\cF_K(\hbx) = A \hbx + \hx \hy \bd +\bb
 = A \left[ \hbx + \hx \hy A^{-1} \bd \right] +\bb
 = A \left[ \hbx + \hx \hy \tbs \right] +\bb,
\end{eqnarray}
where $A$ is a $2\times 2$ matrix and $\bb, \bd$, and $\tbs$ are
two-dimensional vectors given by
\begin{eqnarray*}
A &=& \frac14 \left(\bv_1 -\bv_2 -\bv_3 + \bv_4 ,\bv_1 + \bv_2 - \bv_3 - \bv_4
\right),\\
\bd &=& \frac{\bv_1 -\bv_2 +\bv_3 - \bv_4}{4},\quad \bb =
\frac{\bv_1 +\bv_2 +\bv_3 + \bv_4}{4},\quad \tbs = A^{-1}\bd.
\end{eqnarray*}
Notice that \eqref{eq:cFK} can be understood as the following decomposition of
an affine map and a {\it simple bilinear map} associated with $\tbs$:
$$\cF_K = \cA_K \circ \cS_K,$$
where
$\cA_K:\til{K} \rightarrow K$ and
$\cS_K:\hat{K} \rightarrow \til{K}$ are given by
\begin{equation*}
\cA_K(\tbx) = A \tbx + \bb, \quad
\cS_K(\hbx)=\hbx + \hx\hy\tbs.
\end{equation*}
Here, $\til{K}=\cS_K(\hK)$ is a quadrilateral with four vertices
\begin{equation*}\label{eq:tbv}
\tbv_1=\hbv_1+\tbs,~~\tbv_2 =\hbv_2-\tbs,~~\tbv_3=\hbv_3+\tbs,~~\tbv_4=\hbv_4-
\tbs.
\end{equation*}
%
It should be stressed that the midpoints of $\hat{K}$ are invariant under the map $\cS_K$
and that $\til{K}$ is a perturbation of $\hat{K}$ by a single vector $\tbs$
such that opposite vertices are moved in the same direction (see \figref{fig:Bil}).

The relations of three mappings $\cA_K, \cS_K, \cF_K$ and three
domains
$\hat{K}, \til{K}, K$ can be interpreted
as follows. For given quadrilateral $K\in \Tau_h $ and the reference cube
$\hat{K}$,  $\cF_K$ is a unique bilinear map such that
$\cF_K(\hbv_{j})=\bv_{j} $ for $j=1,2,3,4$. It is easy to see
that there exists a unique simple bilinear map $\cS_K$ and $\til{K}$ such that
$\til{K} = \cS_K(\hat{K})$ and $K = \cA_K(\til{K}).$
%
The intermediate reference domain $\til{K}$ is very useful when we construct
a certain type of basis functions that have specific features in $K$
since $\til{K}$ is connected to the physical domain $K$
by an affine map not by a bilinear map. Adapted to this spirit,
we will construct basis functions in $\til{K}$ instead of $\hat{K}$.

\begin{remark}\label{rem:quad-convex}
Notice that $\til{K}$ is convex if and only if
\begin{eqnarray}\label{quad:convex}
|\ts_1| +|\ts_2| \le 1,
\end{eqnarray}
where the equality holds if and only if $\til{K}$ degenerates to a triangle \cite{park-sheen-morley}.
\end{remark}

\begin{figure}
 \centering
\includegraphics[width=0.60\textwidth]{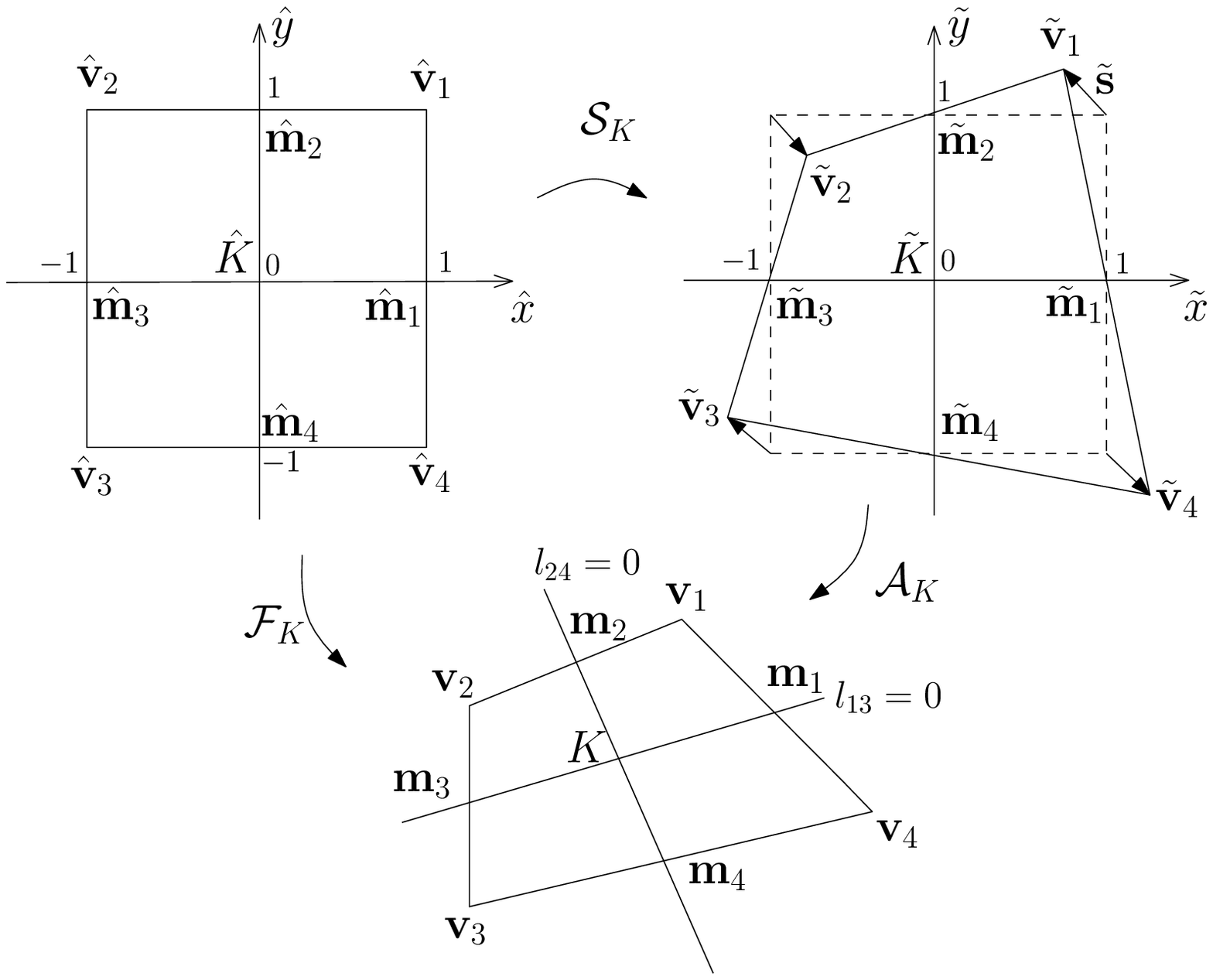}
\caption{A bilinear map $\cF_K$ from $\hat{K}$ to $K$,  a bilinear map $\cS_K$ from $\hat{K}$ to $\til{K}$, and an affine map $\cA_K$ from $\til{K}$ to $K$.}
\label{fig:Bil}
\end{figure}

Our strategy is to use the intermediate reference domain $\til{K}$,
where the ansatz is to set a quartic polynomial similarly to
\eqref{eq:psi-parametric} as follows:
\begin{equation}\label{eq:theta}
\til{\mu}(\tbx)=-\frac{5}{3}\til{\ell}_1(\tbx)\til{\ell}_2(\tbx)\til{\sq}(\tbx),
\end{equation}
where $\til{\ell}_j(\tbx),j=1,2,$ are linear polynomials and
$\til{\sq}(\tbx)$ a quadratic polynomial. We seek a quartic
polynomial $\til{\mu}(\tbx)$ fulfilling the
{\it mean value property} \eqref{eq:mvp} in $\til{K}$.
Naturally, set $\til{\ell}_1(\tbx)$ and $\til{\ell}_2(\tbx)$
to be linear polynomials such that
$\til{\ell}_1(\tbx)=0$ and $\til{\ell}_2(\tbx)=0$
are the equations of lines passing through
$\tbv_1$, $\tbv_3$, and $\tbv_2$, $\tbv_4 $, respectively. Then they are given
(up to multiplicative constants) by
\begin{subeqnarray}\label{eq:cir}
&&\til{\ell}_1(\tbx)=\tx - \ty + \ts_2 - \ts_1,\\
&&\til{\ell}_2(\tbx)=\tx + \ty + \ts_1 + \ts_2.
\end{subeqnarray}

Recall the Gauss quadrature formula:
\begin{eqnarray*}
\int_{-1}^1 f(t)~\operatorname{d}t \approx \frac89 f(0) + \frac59 (f(\xi) +
f(-\xi)),\quad  \xi =\sqrt{\frac35},
\end{eqnarray*}
which is exact for quartic polynomials.
An application of this formula
simplifies the {\it mean value property} \eqref{eq:mvp} into the form
\begin{eqnarray}\label{eq:quadrature-2}
 \til{\mu}(\tbg_{2j-1}) + \til{\mu}(\tbg_{2j}) -2\til{\mu}(\hbm_j) = 0, \quad j = 1, \cdots, 4,
\end{eqnarray}
where
\begin{eqnarray*}
&\tbg_1 = \hbm_1 -\xi(\hbu_2  + \tbs),\quad
\tbg_2 = \hbm_1 +\xi(\hbu_2  + \tbs),\\
&\tbg_3 = \hbm_2 +\xi(\hbu_1  + \tbs),\quad
\tbg_4 = \hbm_2 -\xi(\hbu_1  + \tbs),\\
&\tbg_5 = \hbm_3 +\xi(\hbu_2  - \tbs),\quad
\tbg_6 = \hbm_3 -\xi(\hbu_2  - \tbs),\\
&\tbg_7 = \hbm_4 -\xi(\hbu_1  - \tbs),\quad
\tbg_8 = \hbm_4 +\xi(\hbu_1  - \tbs),
\end{eqnarray*} together with $\hbm_j,j = 1,\cdots,4,$
are the twelve Gauss points on the edges. Here, and in what follows,
we adopt the notations for the standard unit vectors:
$\hbu_1 = \begin{pmatrix} 1 \\ 0 \end{pmatrix}$ and
$\hbu_2 = \begin{pmatrix} 0 \\ 1 \end{pmatrix}.$
Notice that the equations of lines for edges $\tbe_j,j = 1,\cdots,4,$ are given
in vector notation as follows:
\[
\tbe_1(t) = \hbm_1 + t(\hbu_2 + \tbs),\quad
\tbe_2(t) = \hbm_2 + t(\hbu_1 + \tbs),\quad
\tbe_3(t) = \hbm_3 + t(\hbu_2 - \tbs),\quad
\tbe_4(t) = \hbm_4 + t(\hbu_1 - \tbs),
\] for $t\in [-1,1].$
Consider the quartic polynomial \eqref{eq:theta} restricted to an edge
$\tbe_j(t), t\in[-1,1].$ Since $\til{\ell}_1\til{\ell}_2$ is the product of
two linear polynomials which vanishes at the other two end points of each edge,
one sees that
\begin{eqnarray}\label{eq:inter-quad}
\til{\ell}_1(\tbg_{2j-1})\til{\ell}_2(\tbg_{2j-1}) =
\til{\ell}_1(\tbg_{2j})\til{\ell}_2(\tbg_{2j}) = (1-\xi^2)
\til{\ell}_1(\hbm_{j})\til{\ell}_2(\hbm_{j}), \quad \left(\xi = \sqrt{\frac35}\right).
\end{eqnarray}

A combination of \eqref{eq:quadrature-2} and \eqref{eq:inter-quad} yields that
\eqref{eq:mvp} holds if and only if the quadratic polynomial $\til{\sq}$
satisfies
\begin{eqnarray}\label{eq:inter-quad-2}
 \til{\sq}(\tbg_{2j-1}) + \til{\sq}(\tbg_{2j}) -5\til{\sq}(\hbm_j) = 0, \quad j = 1, \cdots, 4.
\end{eqnarray}

A standard use of symbolic calculation gives the general solution of
\eqref{eq:inter-quad-2} in the following form
\begin{eqnarray}\label{eq:sq}
\til{\sq}(\tbx) = \left(\tx + \frac{2}{5} \ts_2\right)^2 +
\left(\ty + \frac{2}{5} \ts_1\right)^2 - \til{r}^2 +
\til{c} \left[(\tx+\frac{2}{5} \ts_2)(\ty+\frac{2}{5} \ts_1) +\frac{6}{25} \ts_1\ts_2 \right],
\end{eqnarray}
with $\til{r} = \frac{\sqrt{6}}{5}\sqrt{ \frac{5}{2} - \cor{\tilde{s}_1^2 - \tilde{s}_2^2}}$
for arbitrary constant $\til{c}\in \mathbb R.$ Here, we assume that the
coefficient of $\tx$ is normalized. Notice that $\til{r}$ \addcor{takes a positive real value}
if $\til{K}$ is convex due to \rmkref{rem:quad-convex}.

Define, for each $\til{c}\in\mathbb R$,
\[
\til{\mu}(\tx,\ty;\til{c}) = -\frac{5}{3}\til{\ell}_1(\tx,\ty)\til{\ell}_2(\tx,\ty)\til{\sq}(\tx,\ty),
\]
where
$\til{\ell}_1$ and $\til{\ell}_2$ are defined by \eqref{eq:cir} and $\til{\sq}$ by
\eqref{eq:sq} depending on $\til{c}$ as well as $\tbs.$

We are now in a position to define a class of {\it nonparametric nonconforming
elements on the intermediate quadrilaterals $\til{K}$} with four degrees of freedom as follows.
\begin{enumerate}
\item $\til{K} = \cS_K(\hat K);$
\item $\til{P}_{\til{K}}(\til{c}) = \Span\{1, \tx, \ty, \til{\mu}(\tx,\ty;\til{c}) \}$;
\item $\til{\Sigma}_{\til{K}} = \{\text{four edge-midpoint values of }
  {\til{K}} \} = \{ \text{four mean values over edges of }
  {\til{K}}            \}$.
\end{enumerate}

By the above construction it is apparent that for any element
$\til{p}\in \til{P}_{\til{K}}(\til{c})$ the {\it mean value property} holds:
\begin{equation*}
\frac{1}{|\til{e}_j|} \int_{\til{e}_j} \til{p}~ d\til{\sigma} =\til{p}(\tbm_j), ~~~~~   j=1,2,3,4.
\end{equation*}
Moreover, the above class of intermediate nonparametric elements is
unisolvent for most of $\til{c}.$
\begin{theorem} \label{thm:thm1}
Assume that $\til{c}$ is chosen such that
$\ts_1^2 + \ts_2^2+\frac{1}{3} + \til{c}~\ts_1\ts_2 \neq 0.$
Then the intermediate nonparametric element
$\left(\til{K},\til{P}_{\til{K}}(\til{c}),\til{\Sigma}_{\til{K}}\right)$ is unisolvent.
\end{theorem}

\begin{proof}
In order to show unisolvency of the space
$\mathrm{Span}\{1,\tx,\ty, \til{\mu}(\tx,\ty;\til{c})\}$ with respect to the degrees of
freedom $f(\hbm_j),j=1,\cdots,4,$
denote the functions $1$, $\tx$, $\ty$, and
$\til{\mu}(\tx,\ty;\til{c})$ by $\til{\phi_1}$, $\til{\phi_2}$, $\til{\phi_3}$, and
$\til{\phi_4}$, respectively and also define $A=(a_{jk}) \in
M_{4\times4}(\mathbb{R})$ by $a_{jk}=\til{\phi_j}(\hbm_k)$.
A symbolic calculation shows that
$\det(A)=16(\ts_1^2 + \ts_2^2+\frac{1}{3} + \til{c}~\ts_1\ts_2),$
from which $A$ is nonsingular for any $\tbs \in \mathbb{R}^2$
if and only if $\til{c}$ is chosen such that
$\ts_1^2 + \ts_2^2+\frac{1}{3} + \til{c}~\ts_1\ts_2 \neq 0.$
This completes the proof.
 \end{proof}

For $\til{c}=0$, the quadratic equation $\til{\sq}(\tbx)=0$
denotes the circle with center
$-\frac25 \begin{pmatrix}\ts_2\\ \ts_1 \end{pmatrix}$
and radius $\til{r}.$ In this case, \eqref{eq:sq} can be easily derived by
a geometric argument as follows. Indeed, assume that $\til{\sq}(\tbx)=0$
denotes the circle with center $\bc=\begin{pmatrix}c_1\\c_2
\end{pmatrix}$ and radius $r$ so that
$\til{\sq}(\tbx) = (\tbx-\bc)\cdot(\tbx-\bc) - r^2.$ Then
\eqref{eq:inter-quad-2} implies that
\begin{eqnarray*}
(\tbg_{2j-1}-\bc)\cdot (\tbg_{2j-1}-\bc)
 + (\tbg_{2j} - \bc)\cdot(\tbg_{2j} - \bc) - 5(\hbm_j-\bc)\cdot(\hbm_j-\bc) = -3r^2, \quad j = 1, \cdots, 4.
\end{eqnarray*}
Arrange these equations as follows:
\begin{eqnarray}\label{eq:inter-quad-3}
(\bc -\tbeta_{2j-1})\cdot (\bc -\tbeta_{2j}) = r^2,\quad j = 1,\cdots,4,
\end{eqnarray}
where the points $\tbeta_{2j-1}$ and $\tbeta_{2j}$ are given between
$\tbg_{2j-1}$ and $\hbm_j$, and $\tbg_{2j}$ and $\hbm_j$, respectively,
explicitly defined as follows: with $\eta=\sqrt{\frac25},$
\begin{eqnarray*}
&\tbeta_1 = \hbm_1 -\eta(\hbu_2  + \tbs),\quad
\tbeta_2 = \hbm_1 +\eta(\hbu_2  + \tbs),\\
&\tbeta_3 = \hbm_2 +\eta(\hbu_1  + \tbs),\quad
\tbeta_4 = \hbm_2 -\eta(\hbu_1  + \tbs),\\
&\tbeta_5 = \hbm_3 +\eta(\hbu_2  - \tbs),\quad
\tbeta_6 = \hbm_3 -\eta(\hbu_2  - \tbs),\\
&\tbeta_7 = \hbm_4 -\eta(\hbu_1  - \tbs),\quad
\tbeta_8 = \hbm_4 +\eta(\hbu_1  - \tbs).
\end{eqnarray*}

Geometrically, \eqref{eq:inter-quad-3} is equivalent to saying that the
location of $\bc$ is such that the four inner products of the vectors
$\bc-\tbeta_{2j-1}$ and $\bc-\tbeta_{2j},$ for $j=1,\cdots, 4,$ are equal.
It is straightforward from the equations
\eqref{eq:inter-quad-3} for  $j=1$ and $j=3$ to see that $c_1 = -\eta^2 \cor{\tilde{s}_2}$,
and similarly from those for $j=2$ and $j=4$ to see that $c_2 = -\eta^2 \cor{\tilde{s}_1}.$
Then $r = \til{r}$ follows immediately.
Thus $\bc$ and $r$ are identical to the center and radius
of the circle represented in \eqref{eq:sq} in the case of $\til{c} = 0.$

\subsubsection{The global nonparametric quadrilateral nonconforming elements}
Turn to the physical domain $K$. It is straightforward to define the finite
elements from $\til{K}$ to $K$ by using the affine map $\cA_K$ which enables
\addcor{the transformed elements} to retain the {\it mean value property} and unisolvency.
A class of {\it nonparametric nonconforming
elements on quadrilaterals $K$} with four degrees of freedom as follows.
\begin{enumerate}
\item $K = \cF_K(\hat K);$
\item $\NC_K=P_{K}(\til{c}) = \Span\{1, x_1, x_2, \mu(x_1,x_2;\til{c}) \}$;
\item $\Sigma_{K} = \{\text{four edge-midpoint values of }
  K \} = \{ \text{four mean values over edges of }
  K \},$
\end{enumerate}
where $\mu(x_1,x_2;\til{c})$ is a quartic polynomial defined by
$\mu(x_1,x_2;\til{c})=\til{\mu}\circ \cA_K^{-1}(x_1,x_2;\til{c})
=-\frac{5}{3}\ell_1(x_1,x_2)\ell_2(x_1,x_2) q(x_1,x_2;\til{c})$, with
\begin{eqnarray*}
\ell_1(\bx)=\til{\ell}_1 \circ \cA_K^{-1}(\bx),\quad
\ell_2(\bx)=\til{\ell}_2 \circ \cA_K^{-1}(\bx),\quad
q(\bx;\til{c})=\til{\sq} \circ \cA_K^{-1}(\bx).
\end{eqnarray*}

Notice that $\mu(\bx;\til{c})$ can be interpreted as a product of
two linear polynomials and one quadratic polynomial such that the straight lines
$\ell_1(\bx)=0$ and  $\ell_2(\bx)=0$ are passing through $\bv_1$, $\bv_3$ and
$\bv_2$, $\bv_4$, respectively and $q(\bx;\til{c})=0$
is an ellipse which is determined to satisfy the {\it mean value properties} for
$\til{\mu}(\tbx)$.

We now define the global nonparametric DSSY element spaces as follows
\begin{equation*}
\begin{split}
&\NC^{np}_{h}=\{ v_h \in L^2(\O)  ~|~ v_h|_K \in \NC_K ~\mathrm{for}~ K \in \Tau_h, v_h ~\mathrm{is}~ \mathrm{continuous}~ \mathrm{at}~ \mathrm{the}~ \mathrm{midpoint}~ \mathrm{of}~ \mathrm{each}~ e \in \cE_h  \},   \\
&\NC^{np}_{h,0}=\{v_h \in \NC^{np}_{h} ~|~v_h~\mathrm{is}~ \mathrm{zero}~ \mathrm{at}~ \mathrm{the}~ \mathrm{midpoint}~ \mathrm{of}~ \mathrm{each}~ e \in \cE_h \cap \partial \O \}.
\end{split}
\end{equation*}

\begin{remark}
Since these new finite element spaces have the orthogonal property as in
\cite{dssy-nc-ell}, clearly the optimal convergence order is guaranteed for
solving second-order elliptic problems.
\cor{Indeed, \eqref{eq:mvp} implies the pass of a patch test against constant functions on
each interior edge (see (2.7a) and (2.7b) of \cite{dssy-nc-ell}), which in turn implies the following bound of the
consistent error term in the second Strang lemma:
\[
\sup_{w_h\in\NC^{np}_{h,0} } \frac{\left|a_h(u,w_h)- (f,w_h) \right|
}{\|w_h\|_{1,h}} \le C \|u\|_2 h,
\]
where $u\in H^2(\O)\cap H_0^1(\O)$ is a solution to $a(u,v) = (f,v)~ \forall v
\in H^1_0(\O),$
and $a(\cdot,\cdot)$ and $a_h(\cdot,\cdot)$ are bounded, coercive bilinear
forms
on  $H^1_0(\O)$ and $\NC^{np}_{h,0}$, respectively.}
\end{remark}

\begin{remark}
The new nonparametric DSSY elements will be used as a stable family of mixed finite elements for
the velocity fields, combined
with the piecewise constant element for pressure, in solving the Navier-Stokes
equations \cite{cdy99, jeon-nam-sheen-shim-opt, rann}. The nonconforming nature \addcor{enables us} to solve elasticity problems without
numerical locking, either \cite{lls-nc-elast, zhiminzhang}. \cor{See the numerical
experiments in Subsections 3.2 and 3.3.}
\end{remark}

\begin{remark}
In practice, the choice $\til{c} = 0$ is recommended since it minimizes
the number of computations in applying quadrature rules.
\end{remark}

\begin{remark}
\addcor{One may construct basis functions in a sixth-degree polynomial
space other than the quartic polynomial as in \eqref{DSSY-psi}
following the same idea. However, using a
higher-degree polynomial space requires a higher accuracy quadrature rule in
the construction of the stiffness matrix. 
In this sense, the quartic polynomial space seems to be a reasonable choice in view of implementation issues.}
\end{remark}

\section{Numerical results}
\cor{\subsection{The elliptic problem}}
In this section we perform numerical experiments for a simple elliptic
problem:
\begin{eqnarray*}
-\Delta u  &=& f \qquad \mathrm{in}~ \O, \\
u &=& 0 \qquad \mathrm{on}~  \partial \O,
\end{eqnarray*}
on the domain  $\O=(0,1)^2$.
The source function $f$ is given so that the exact solution is
\begin{equation*}
u(\bx)=\sin\pi x_1\sin\pi x_2. 
\end{equation*}

We consider two kinds of elements:
the parametric DSSY element  $\NC^{p}_{h,0}$, and
the nonparametric DSSY elements $\NC^{np}_{h,0}$ with $\til{c}=0$ and
$\til{c}=1.$
Also two types of quadrilateral meshes were employed:
uniformly $\theta$-dependent quadrilateral meshes as shown in
\figref{fig:domain} are used and the  randomly perturbed quadrilateral meshes
depicted in \figref{fig:nonuniform}.
The uniformly $\theta$-dependent quadrilaterals become
rectangles if $\theta=0$, while they degenerate into triangles if
$\theta=1.$
\begin{figure}[ht]
 \centering
\includegraphics[width=0.60\textwidth]{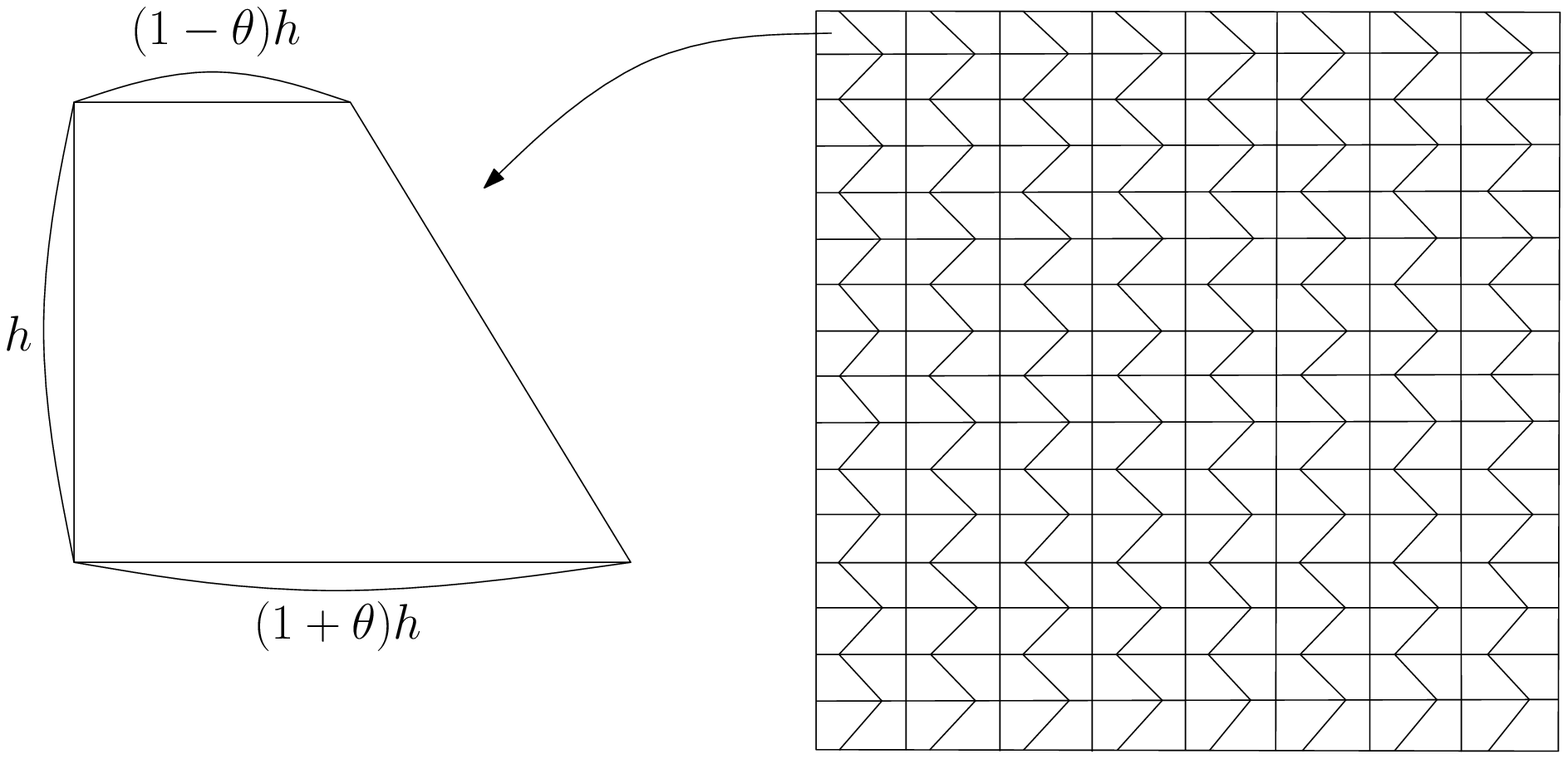}
\caption{A uniform trapezoidal triangulation with a trapezoidal with parameter $0 \leq \theta < 1$.}
\label{fig:domain}
\end{figure}

\begin{figure}[ht]
 \centering
\includegraphics[width=0.60\textwidth]{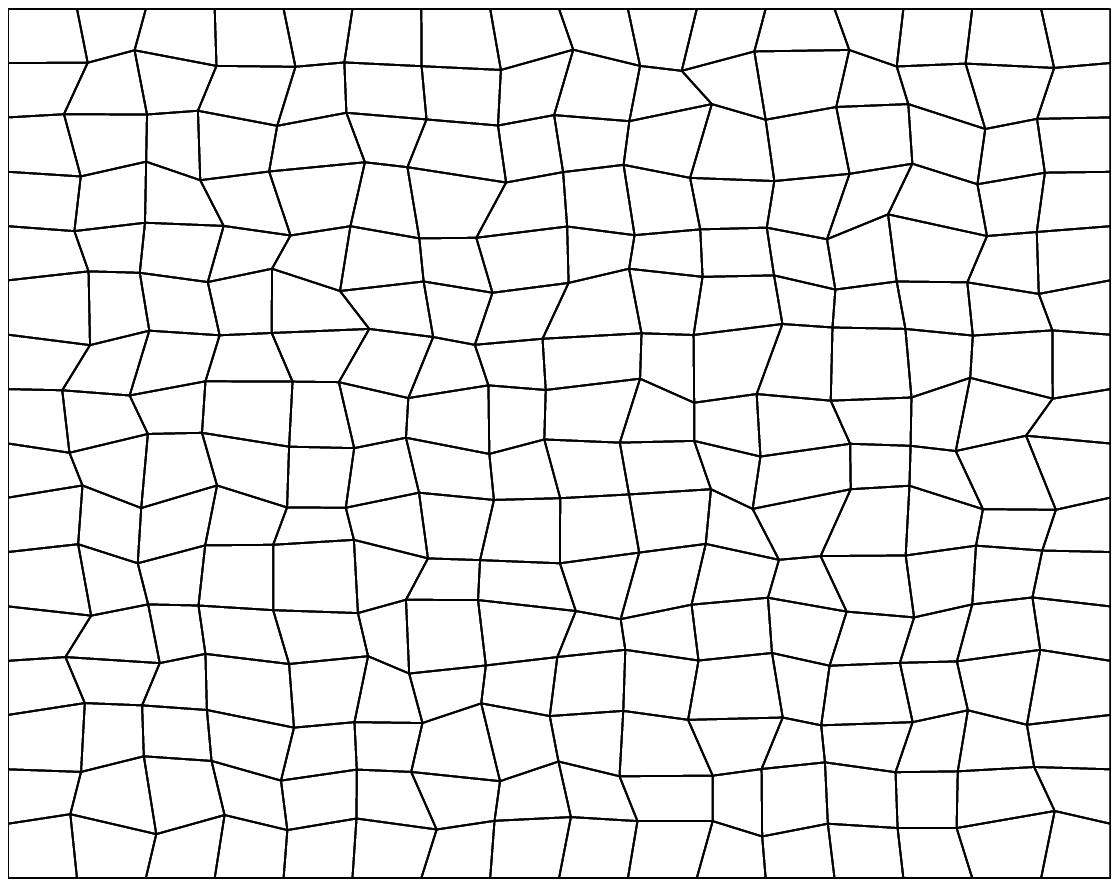}
\caption{A nonuniform randomly perturbed quadrilateral triangulation}
\label{fig:nonuniform}
\end{figure}

The tables containing numerical results are organized as follows:
the parametric nonconforming elements in \tabrefs{table:ell1}{table:ell4},
the nonparametric nonconforming elements with $\til{c} = 0$ in \tabrefs{table:ell2}{table:ell5},
and those with $\til{c} = 1$ in \tabrefs{table:ell3}{table:ell6};
the uniformly $\theta$-dependent trapezoidal meshes in \tabrefss{table:ell1}{table:ell3}
and the nonuniform quadrilateral meshes in \tabrefss{table:ell4}{table:ell6}.

We tested several different $\theta$'s, but the convergence behaviors were
quite similar and thus we report only the case of $\theta=0.7$.

Numerical experiments were
performed with increasing values of $\til{c}$ such as 10, 100, 1000, and so
on. The larger $\til{c}$ values are chosen, the slower convergence is observed.
We only present numerics for the two nonparametric elements with $\til{c}=0$
and 1. At this point we recommend readers to use $\til{c} = 0$ for its
simplicity.

As observed in the uniform mesh the convergence order is optimal for both
elements and the values of numerical solutions are almost identical.  \cor{In order
to compare cost efficiency in a fair fashion, we computed nonparametric basis functions
for each quadrilateral
and applied the static condensation to circumvent bubble
functions for parametric element also for each quadrilateral.
From \tabref{table:ell7} we observe that
when the mesh size $h$ is larger than 1/100, the nonparametric element
is cheaper to use; however, the computing time ratios approach to 1
(still the use of nonparametric element seems to be cheaper),
as the mesh size tends to decrease.}
\addcor{These phenomena are perhaps due to the fact that the additional cost in static
condensation for the parametric elements takes a less portion in the total
computing time as the mesh size decreases.}

\begin{table}[ht]
  \begin{center}
\begin{tabular}{||c||c|c|c|c|c|}\hline
h & DOF & $||u-u_h||_{0,\O}$ & ratio  &$||u-u_h||_{1,h}$  & ratio  \\ \hline
1/4   & 40     &  0.5284E-01 & -  &0.8532  &-           \\ \hline
1/8   & 176    & 0.1556E-01  &1.76  &0.4458 & 0.94  \\ \hline
1/16  & 736    &0.4184E-02  &1.89  &0.2274 &0.97    \\ \hline
1/32  & 3008   & 0.1096E-02 &1.93  &0.1147 &0.99     \\ \hline
1/64  & 12160  & 0.2810E-03  &1.96  &0.5756E-01 &0.99    \\ \hline
1/128  & 48896 &0.7117E-04  &1.98  & 0.2883E-01  &1.00   \\ \hline
1/256  &196096 &0.1791E-04  &1.99  & 0.1443E-01 &1.00      \\ \hline
\end{tabular}
  \end{center}
 \caption{\label{table:ell1} Computational results for $\NC^{p}_{h,0}$ with $\theta=0.7$ for the elliptic problem. }
\end{table}

\begin{table}[ht]
  \begin{center}
\begin{tabular}{||c||c|c|c|c|c|}\hline
h & DOF & $||u-u_h||_{0,\O}$ & ratio &$||u-u_h||_{1,h}$ & ratio  \\ \hline
1/4   &24 &0.5437E-01 &- &0.8221 &-  \\ \hline
1/8   &112 &0.1568E-01 &1.79 &0.4302 & 0.93   \\ \hline
1/16  &480 &0.4145E-02 &1.92 &0.2213 &0.96   \\ \hline
1/32   &1984 & 0.1084E-02 &1.93 &0.1124  &0.98    \\ \hline
1/64   & 8064 &0.2788E-03 &1.96 &0.5659E-01 &0.99   \\ \hline
1/128   & 32512 &0.7077E-04 &1.98 & 0.2839E-01 &1.00  \\ \hline
1/256  &130560    &0.1783E-04  &1.99  & 0.1422E-01 &  1.00    \\ \hline
\end{tabular}
  \end{center}
 \caption{\label{table:ell2} Computational results for $\NC^{np}_{h,0}$ with $\theta=0.7$ and $\til{c}=0$ for the elliptic problem}
\end{table}

\begin{table}[ht]
  \begin{center}
\begin{tabular}{||c||c|c|c|c|c|}\hline
h & DOF & $||u-u_h||_{0,\O}$ & ratio &$||u-u_h||_{1,h}$ & ratio  \\ \hline
1/4   &24 &0.5840E-01 &- &0.8486 &-  \\ \hline
1/8   &112 &0.1655E-01 &1.82 &0.4452 & 0.93   \\ \hline
1/16  &480 &0.4229E-02 &1.97 &0.2261 &0.98   \\ \hline
1/32   &1984 &0.1102E-02 &1.94 &0.1145 &0.98    \\ \hline
1/64   & 8064 &0.2836E-03 &1.96 &0.5760E-01 &0.99   \\ \hline
1/128   & 32512 &0.7212E-04 &1.98 &0.2887E-01 &1.00  \\ \hline
1/256  &130560    &0.1819E-04 &1.99  &0.1446E-01 &  1.00    \\ \hline
\end{tabular}
  \end{center}
 \caption{\label{table:ell3} Computational results for $\NC^{np}_{h,0}$ with $\theta=0.7$ and $\til{c}=1$ for the elliptic problem. }
\end{table}


\begin{table}[ht]
  \begin{center}
\begin{tabular}{||c||c|c|c|c|c|}\hline
h & DOF & $||u-u_h||_{0,\O}$ & ratio  &$||u-u_h||_{1,h}$  & ratio \\ \hline
1/4   & 40     &0.3490E-01&- &0.7183 &-   \\ \hline
1/8   & 176    &0.8663E-02 &2.01 &0.3657  & 0.97 \\ \hline
1/16  & 736    &0.2287E-02 &1.92 &0.1871  &0.97  \\ \hline
1/32  & 3008   &0.5835E-03 &1.97  & 0.9387E-01 &0.99  \\ \hline
1/64  & 12160  &0.1481E-03 &1.98 &0.4721E-01 & 0.99 \\ \hline
1/128  & 48896 &0.3729E-04 & 1.99 &0.2363E-01 &1.00 \\ \hline
1/256  &196096 &0.9350E-05 &2.00 &0.1183E-01 &1.00   \\ \hline
\end{tabular}
  \end{center}
 \caption{\label{table:ell4} Computational results for $\NC^{p}_{h,0}$ on the
nonuniform randomly perturbed meshes for the elliptic problem.}
\end{table}

\begin{table}[ht]
  \begin{center}
\begin{tabular}{||c||c|c|c|c|c|}\hline
h & DOF & $||u-u_h||_{0,\O}$ & ratio &$||u-u_h||_{1,h}$ & ratio  \\ \hline
1/4   &24 & 0.3594E-01 &- &0.7363  &- \\ \hline
1/8   &112 &0.8760E-02 &2.04 & 0.3682  &1.00   \\ \hline
1/16  &480  &0.2290E-02 &1.94  &0.1873 &0.98  \\ \hline
1/32   &1984 &0.5834E-03  &1.97 &0.9386E-01 &1.00    \\ \hline
1/64   & 8064 &0.1479E-03  &1.98 &0.4718E-01 &0.99  \\ \hline
1/128   & 32512 &0.3725E-04  &1.99 &0.2362E-01 &1.00  \\ \hline
1/256  &130560    &0.9341E-05   &2.00 &0.1182E-01 &1.00     \\ \hline
\end{tabular}
  \end{center}
 \caption{\label{table:ell5} Computational results for $\NC^{np}_{h,0}$ on the
  nonuiform randomly perturbed meshes when $\til{c}=0$ for the elliptic problem. }
\end{table}

\begin{table}[ht]
  \begin{center}
\begin{tabular}{||c||c|c|c|c|c|}\hline
h & DOF & $||u-u_h||_{0,\O}$ & ratio &$||u-u_h||_{1,h}$ & ratio  \\ \hline
1/4   &24 & 0.3598E-01 &- & 0.7370  &- \\ \hline
1/8   &112 &0.8752E-02 &2.04 & 0.3684  &1.00   \\ \hline
1/16  &480  &0.2290E-02 &1.93  &0.1874  &0.98  \\ \hline
1/32   &1984 & 0.5842E-03  &1.97 &0.9398E-01 &1.00    \\ \hline
1/64   & 8064 &0.1481E-03 &1.98 &0.4725E-01 &0.99  \\ \hline
1/128   & 32512 &0.3730E-04  &1.99 &0.2365E-01 &1.00  \\ \hline
1/256  &130560    &0.9353E-05   &2.00 &0.1184E-01 &1.00     \\ \hline
\end{tabular}
  \end{center}
 \caption{\label{table:ell6} Computational results for $\NC^{np}_{h,0}$ on the
   nonuniform randomly perturbed meshes when $\til{c}=1$ for the elliptic problem. }
\end{table}

\begin{table}[ht]
  \begin{center}
\begin{tabular}{||c||c|c|c|c|c|c|}\hline
h   & $\theta=0.3 $ & $\theta=0.5 $ & $\theta=0.7 $ & Random mesh  \\ \hline
{1/8}  & {0.6764}  & {0.6764}   & {0.6666} & {0.6571} \\ \hline
{1/16} & {0.6711}  & {0.6621}   & {0.6802} & {0.6712}  \\ \hline
{1/32} & {0.6796}  & {0.6761}   & {0.6844} & {0.7022}  \\ \hline
{1/64} & {0.7333}  & {0.7285}   & {0.7303} & {0.7344}  \\ \hline
{1/128}& {0.7611}  & {0.7656}   & {0.7540} & {0.7275}  \\ \hline
{1/256}& {0.8136}  & {0.8296}   & {0.7924} & {0.7875}  \\ \hline
{1/512}& {0.9431}  & {0.9170}   & {0.8861} & {0.8415}   \\ \hline
\end{tabular}
  \end{center}
 \caption{\label{table:ell7} Ratio of computing time
   t($\NC^{np}_{h,0}$)/t$(\NC^{p}_{h,0})$ for the elliptic problem
on uniform trapezoidal meshes with varying parameter $\theta$ and on
nonuniform randomly perturbed meshes. }
\end{table}

\subsection{The incompressible Stokes equations}
In this subsection, we apply $\mathcal{NC}^{np}_{h,0}$ to approximate each
component of the velocity fields in solving the
incompressible Stokes equations in two dimensions, while the piecewise
constant element is employed to approximate the pressure.

Set $\O=(0,1)^2$ and consider the following Stokes equations:
\begin{equation*}
\begin{split}
-\Delta \mathbf{u} + \nabla p &= \mathbf{f} \qquad \mathrm{in}~ \O, \\
\nabla \cdot \mathbf{u} &= 0 \qquad \mathrm{in} ~\O, \\
\mathbf{u} &= 0 \qquad \mathrm{on}~  \partial \O,
\end{split}
\end{equation*}
where the force term $\mathbf f$ is generated by the following exact solution
\begin{equation*}
\begin{split}
\mathbf{u}(x_1,x_2)&= \left( \begin{array}{c} e^{x_1+2x_2}(x_1^4-2x_1^3+x_1^2)(2x_2^4-4x_2^2+2x_2)  \\    -e^{x_1+2x_2}(x_1^4+2x_1^3-5x_1^2+2x_1)(x_2^4-2x_2^3+x_2^2)
 \end{array} \right),\\
p(x_1,x_2)&=-\sin2\pi x_1\sin2\pi x_2.
\end{split}
\end{equation*}

\tabref{table:St1} shows the numerical results on uniform trapezoidal meshes with $\theta=0.7$ and $\til{c} = 0.$
Similarly, \tabref{table:St3} presents the results on the perturbed
nonuniform meshes with $\til{c} = 0.$
From these numerical results, we observe the optimal
convergence rates of $O(h^2)$ and $O(h)$ for the velocity and pressure in
$L^2$ norm, respectively. The numerical solutions in the case with $\til c\neq
0$ behave similarly, whose tables are omitted to report.

\begin{table}[ht]
  \begin{center}
\begin{tabular}{||c||c|c|c|c|c|}\hline
{h} & {DOF} & {$||\mathbf{u}-\mathbf{u}_h||_{0,\O}$} &{ ratio} &{$||p-p_h||_{0,\O}$} & {ratio}  \\ \hline
{1/4}   &{63} & {0.1302E-01} &{-} & {0.2770} &{-}  \\ \hline
{1/8}   &{287} & {0.5424E-02}  &{1.26} & {0.1898}  &{0.55}    \\ \hline
{1/16}  &{1215} & {0.1631E-02} &{1.73} &{0.9571E-01} &{0.99}   \\ \hline
{1/32}   &{4991} &{0.4396E-03}  &{1.89} &{0.4801E-01} &{1.00}  \\ \hline
{1/64}   & {20223} &{0.1130E-03}  &{1.96} &{0.2410E-01}  &{0.99}   \\ \hline
{1/128}   & {81407} & {0.2855E-04}  &{1.98} & {0.1208E-01} &{1.00}  \\ \hline
\end{tabular}
  \end{center}
 \caption{\label{table:St1} Computational results for $\NC^{np}_{h,0}$ with $\theta=0.7$ and $\til{c}=0$ for the Stokes problem. }
\end{table}

\begin{table}[ht]
  \begin{center}
\begin{tabular}{||c||c|c|c|c|c|}\hline
{h} & {DOF} & {$||\mathbf{u}-\mathbf{u}_h||_{0,\O}$} &{ ratio} &{$||p-p_h||_{0,\O}$} & {ratio}  \\ \hline
{1/4}   &{63} & {0.1205E-01} &{-} &{0.2960} &{-}  \\ \hline
{1/8}   &{287} & {0.3474E-02}  &{1.79} &{0.1635}  &{0.85}        \\ \hline
{1/16}  &{1215} &  {0.9061E-03} &{1.94} & {0.8381E-01} &{0.96}   \\ \hline
{1/32}   &{4991} &{0.2332E-03} &{1.96}  &{0.4216E-01} &{0.99}  \\ \hline
{1/64}   & {20223} &  {0.5801E-04} &{2.00}  &{0.2103E-01} &{ 1.00}  \\ \hline
{1/128}   & {81407} &{0.1455E-04}   &{2.00}  &{0.1052E-01}  &{ 1.00}  \\ \hline
\end{tabular}
  \end{center}
 \caption{\label{table:St3} Computational results for $\NC^{np}_{h,0}$ on the
   perturbed nonuniform mesh when $\til{c}=0$ for the Stokes problem. }
\end{table}

\subsection{The planar linear elasticity problem}
In this subsection,  the nonparametric element $\mathcal{NC}^{np}_{h,0}$
is applied to approximate each component of the displacement fields
for the planar linear elasticity problem with the clamped boundary condition.

Set $\O=(0,1)^2.$
For $(\mu,\lambda) \in [\mu_0,\mu_1]\times [\lambda_1,\infty),$ consider the
  following elasticity equations with homogeneous boundary condition:
\begin{equation*}
\begin{split}
-(\lambda+\mu)\nabla(\nabla \cdot \mathbf{u}) -\mu \Delta \mathbf{u} &=\mathbf{f}\qquad \mathrm{in}~ \O, \\
\mathbf{u} &= 0 \qquad \mathrm{on}~  \partial \O,
\end{split}
\end{equation*}
where the external force term $\mathbf f$ is generated by the following
exact solution
\begin{equation*}
\begin{split}
u_1(x_1,x_2)&= \sin2\pi x_2(-1+\cos2\pi x_1) +\frac{1}{1+\lambda}\sin\pi x_1 \sin \pi x_2, \\
u_2(x_1,x_2)&= -\sin2\pi x_1(-1+\cos2\pi x_2) +\frac{1}{1+\lambda}\sin\pi x_1 \sin \pi x_2.
 \end{split}
\end{equation*}

In order to check numerical locking phenomena, the Lam\'e parameters are
chosen such that  $(\mu,\lambda)=(1,1)$ and $(1,10^5)$. The
numerical results are presented in \tabrefs{table:ela1}{table:ela2} for both
cases on uniform trapezoidal meshes with $\theta=0.7$ and $\til{c} = 0.$
Similar results are given in \tabrefs{table:ela3}{table:ela4} for both cases
on the randomly perturbed meshes with $\til{c} = 0.$
One can easily observe from the numerical results that the nonparametric element
$\mathcal{NC}^{np}_{h,0}$ can be used to solve
planar elasticity problems with the clamped boundary condition optimally without
numerical locking.

\begin{table}[ht]
  \begin{center}
\begin{tabular}{||c||c|c|c|c|c|}\hline
{h} & {DOF} & {$||\mathbf{u}-\mathbf{u}_h||_{0,\O}$} & {ratio} &{$||\mathbf{u}-\mathbf{u}_h||_{1,h}$} & {ratio}  \\ \hline
{1/4}   &{48} &{0.3787}  &{-} &{5.640} &{-}  \\ \hline
{1/8}   &{224} &{0.1074} &{1.81} &{2.941} &{ 0.93}    \\ \hline
{1/16}  & {960} &{0.2911E-01} &{1.88} &{1.524} &{ 0.94}  \\ \hline
{1/32}   &{3968} &{0.7521E-02} & {1.95}&{0.7712} &{0.98}    \\ \hline
{1/64}   &{16128} &{0.1906E-02} & {1.98}&{0.3870}  &{ 0.99}  \\ \hline
{1/128}   &{65024} &{0.4793E-03} &{1.99} &{  0.1937}&{1.00}  \\ \hline
\end{tabular}
  \end{center}
 \caption{{ \label{table:ela1} Computational results for $\NC^{np}_{h,0}$ with $\theta=0.7$, $\til{c}=0$, $\mu=1$, and $\lambda=1$   for the elasticity problem.} }
\end{table}

\begin{table}[ht]
  \begin{center}
\begin{tabular}{||c||c|c|c|c|c|}\hline
{h} & {DOF} & {$||\mathbf{u}-\mathbf{u}_h||_{0,\O}$} & {ratio} &{$||\mathbf{u}-\mathbf{u}_h||_{1,h}$} & {ratio}  \\ \hline
{1/4}   &{48} &{0.3781} &{-} &{5.631} &{-}  \\ \hline
{1/8}   &{224} &{ 0.1075} &{1.81} &{2.918}  &{0.95}    \\ \hline
{1/16}  & {960} &{ 0.2900E-01} &{1.89} &{1.511} &{0.95}   \\ \hline
{1/32}   &{3968} &{0.7495E-02} &{1.95} &{0.7642} &{0.98}    \\ \hline
{1/64}   &{16128} &{0.1902E-02} &{1.98} &{0.3834} & {0.99}   \\ \hline
{1/128}   &{65024} & {0.4789E-03} &{1.99} &{0.1919} & {1.00} \\ \hline
\end{tabular}
  \end{center}
 \caption{{\label{table:ela2} Computational results for $\NC^{np}_{h,0}$ with $\theta=0.7$, $\til{c}=0$, $\mu=1$, and $\lambda=10^5$   for the elasticity problem. }}
\end{table}

\begin{table}[ht]
  \begin{center}
\begin{tabular}{||c||c|c|c|c|c|}\hline
{h} & {DOF} & {$||\mathbf{u}-\mathbf{u}_h||_{0,\O}$} & {ratio} &{$||\mathbf{u}-\mathbf{u}_h||_{1,h}$} & {ratio}  \\ \hline
{1/4}   &{48} &{0.2517} &{-} & {4.679} &{-}  \\ \hline
{1/8}   &{224} &{0.6530E-01} &{1.77} &{2.459} &{0.73}   \\ \hline
{1/16}  &{ 960} &{0.1724E-01} &{1.92} & {1.264} &{0.96}   \\ \hline
{1/32}   &{3968} &{ 0.4392E-02}& { 1.97}& {0.6382}&{ 0.99}   \\ \hline
{1/64}   &{16128} & {0.1105E-02}&{ 1.99} &{0.3197}  & {1.00}  \\ \hline
{1/128}   &{65024} &{ 0.2776E-03} &{1.99} & {0.1601}  & {1.00} \\ \hline
\end{tabular}
  \end{center}
 \caption{{\label{table:ela3} Computational results for $\NC^{np}_{h,0}$  on the perturbed nonuniform mesh with $\til{c}=0$, $\mu=1$, and $\lambda=1$   for the elasticity problem. }}
\end{table}

\begin{table}[ht]
  \begin{center}
\begin{tabular}{||c||c|c|c|c|c|}\hline
{h} & {DOF} & {$||\mathbf{u}-\mathbf{u}_h||_{0,\O}$} & {ratio} &{$||\mathbf{u}-\mathbf{u}_h||_{1,h}$} & {ratio}  \\ \hline
{1/4}   &{48} &{ 0.2523} &{-} & {4.659} &{-}  \\ \hline
{1/8}   &{224} &{0.6591E-01} &{1.93} &{2.444} &{0.93}    \\ \hline
{1/16}  &{ 960} & {0.1746E-01} &{1.92} &{1.255} &{0.96}   \\ \hline
{1/32}   &{3968} &{0.4461E-02} &{1.97} &{0.6334} &{0.99}   \\ \hline
{1/64}   &{16128} & {0.1124E-02} &{1.99} &{0.3174}  & {1.00}  \\ \hline
{1/128}   &{65024} &{0.2825E-03} &{1.99} &{0.1589} & {1.00} \\ \hline
\end{tabular}
  \end{center}
 \caption{{\label{table:ela4} Computational results for $\NC^{np}_{h,0}$ on the
   perturbed nonuniform mesh with $\til{c}=0$, $\mu=1$, and $\lambda=10^5$   for the elasticity problem.} }
\end{table}

\section*{Acknowledgments}
The research of YJ is supported by NRF of Korea (No. 2010-0021683). This research was supported by NRF of Korea(No. 2012-0000153).


\bibliographystyle{abbrv}


\end{document}